\documentclass[12pt, reqno]{amsart}
\usepackage{amssymb, amsthm, amsmath, amsfonts}
\usepackage{array, epsfig}
\usepackage{bbm}
\usepackage{hyperref}
\usepackage[numbers,square]{natbib}
\usepackage{color}

  \evensidemargin
\oddsidemargin
\newcommand{\bB}{\mathbb{B}}

\newcommand{\bS}{\mathbb{S}}

\newcommand{\cF}{\mathcal{F}}

\newcommand{\E}{\mathbb E}

\newcommand{\R}{\mathbb{R}}

\renewcommand{\P}{\mathbb{P}}

\newcommand{\eps}{\varepsilon}

\newcommand{\ind}{\mathbbm{1}}

\newcommand{\dd}{{\rm d}}

\theoremstyle{plain}
\newtheorem{theorem}{Theorem}[section]
\newtheorem{lemma}[theorem]{Lemma}

\theoremstyle{definition}

\theoremstyle{remark}

\begin{document}

\author{Zakhar Kabluchko}
\address{Zakhar Kabluchko: Institut f\"ur Mathematische Stochastik,
Westf\"alische Wilhelms-Universit\"at M\"unster,
Orl\'eans-Ring 10,
48149 M\"unster, Germany}
\email{zakhar.kabluchko@uni-muenster.de}

\title[Angle sums of random simplices in dimensions $3$ and $4$]{Angle sums of random simplices in dimensions $3$ and $4$}

\keywords{Random polytopes, random simplices, solid angles, sum of angles, beta distributions}

\subjclass[2010]{Primary: 52A22, 60D05; Secondary: 52B11.}

\begin{abstract}
Consider a random $d$-dimensional simplex whose vertices are $d+1$ random points sampled independently and uniformly from the unit sphere in $\mathbb R^d$. We show that the expected sum of solid angles at the vertices of this random simplex equals $\frac 18$ if $d=3$ and $\frac{539}{288\pi^2}-\frac 16$ if $d=4$.
The angles are measured as proportions of the full solid angle which is normalized to be $1$.
Similar formulae are obtained if the vertices of the simplex are uniformly distributed in the unit ball.
These results are special cases of general formulae for the expected angle-sums of  random beta simplices in dimensions $3$ and $4$.
\end{abstract}

\maketitle

\section{Main results}

\subsection{Introduction and notation}
The sum of the measures of angles in any plane triangle is constant and equals $1/2$ of the full plane angle. For a tetrahedron in a three-dimensional space, neither the sum of the solid three-dimensional angles at its vertices, nor the sum of the solid dihedral angles at its edges is constant. In fact, the former can take any value between $0$ and $1/2$ of the full angle, whereas the latter can take any value  between $1$ and $3/2$. The range of all possible values of angle sums, for simplices of arbitrary dimension and for angles taken at faces of arbitrary dimension, was completely identified by~\citet[(24) on pp.~208--209]{perles_shephard}.

The aim of the present paper is to prove an explicit formula for expected angle-sums of \textit{random} simplices whose vertices are independent and identically distributed random points sampled according to the uniform distribution on the unit sphere or the unit ball in dimensions $3$ and $4$.  These two distributions are special cases of a general family of beta distributions for which we shall also provide an explicit formula.

Let us first introduce the necessary notation, referring to the book by Schneider and Weil~\cite{SW08} for an extensive account of stochastic geometry.  Let  $\|x\| = (x_1^2+\ldots+x_d^2)^{1/2}$ be the Euclidean norm of the vector $x= (x_1,\ldots,x_d)\in\R^d$ and denote the unit ball and the unit sphere in $\R^d$ by
$$
\bB^{d}:=\{x\in\R^d\colon \|x\| \leq 1\}
\quad
\text{and}
\quad
\bS^{d-1} := \{x\in \R^d\colon \|x\| = 1\}.
$$
Let $P\subset \R^d$ be a $d$-dimensional convex polytope. Denote by $\cF_0(P)$ the set of vertices of $P$.
The \textit{internal angle} of $P$ at its vertex $x_0$ is defined as
$$
\beta(x_0,P) := \P [\exists \eps>0 \text{ such that } x_0 + \eps U \in P],
$$
where $U$ is a random vector having the uniform distribution on the unit sphere $\bS^{d-1}$.
Finally, denote the sum of angles of $P$ at its vertices by
$$
s_0(P) = \sum_{x_0 \in \cF_0(P)} \beta(x_0,P).
$$
Note that the units of measurement for solid angles were chosen such that the full solid angle has measure $1$.


Consider $d+1$ random points $X_0,\ldots,X_d$ in $\R^d$ drawn independently according to some probability distribution $\mu$. Define the random simplex $T$ as their convex hull:
$$
T := [X_0,\ldots,X_d]
:= \{\lambda_0X_0+\ldots+\lambda_d X_d\colon \lambda_0+\ldots+\lambda_d=1, \lambda_0\geq 0,\ldots, \lambda_{d}\geq 0\}.
$$
We are interested in determining the expected value of the angle sum $s_0(T)$.
One special case is already known: If $\mu$ is the standard Gaussian distribution on $\R^d$, then $\E s_0(T)$ coincides with the sum of angles of the regular $d$-dimensional simplex at its vertices; see~\cite{kabluchko_zaporozhets_gauss_simplex}. In fact, the same conclusion applies if $\mu$ is any Gaussian distribution with non-singular covariance matrix~\cite{kabluchko_zaporozhets_ongoing}.

\subsection{Main results}
The aim of the present paper is to prove the following theorems.
\begin{theorem} [Points sampled on the sphere]\label{theo:main_sphere}
Let $X_0,\ldots,X_d$ be independent random points sampled uniformly on the unit sphere $\bS^{d-1}= \{x\in\R^d\colon \|x\| = 1\}$. Then, for $d=3$ and $d=4$,  the expected angle-sum of the simplex $T = [X_0,\ldots,X_d]$ is given by
$$
\E  s_0(T) =
\begin{cases}
\frac 18, &\text{ if } d=3,\\
\frac{539}{288\pi^2}-\frac 16, &\text{ if } d=4.
\end{cases}
$$
\end{theorem}
\begin{theorem} [Points sampled in the ball]\label{theo:main_ball}
Let $X_0,\ldots,X_d$ be independent random points sampled uniformly in the unit ball  $\bB^{d}=\{x\in\R^d\colon \|x\| \leq 1\}$. Then, for $d=3$ and $d=4$, the expected angle-sums of the simplex $T = [X_0,\ldots,X_d]$ are given by
$$
\E  s_0(T) =
\begin{cases}
\frac {401}{2560}, &\text{ if } d=3,\\
\frac{1692197}{846720 \pi^2}-\frac 16, &\text{ if } d=4.
\end{cases}
$$
\end{theorem}

These theorems are particular cases of Theorems~\ref{theo:general_beta_d=3} and~\ref{theo:general_beta_d=4} which we shall state below. The proofs of the latter theorems will be given in Section~\ref{sec:proofs}.
As we shall explain at the end of Section~\ref{sec:proofs}, the proofs do not carry over to dimensions $d\geq 5$. Using different methods, it is possible to build an algorithm computing the expected angle-sums in higher dimensions. This problem will be studied in a separate paper. The advantage of the method used in the present paper is its simplicity.


For the sake of brevity, we considered only angle sums at vertices of the simplex. More generally, we can denote by $s_k(T)$ the sum of the internal angles of $T$ at all faces of dimension $k\in \{0,\ldots,d-1\}$. Luckily, in dimensions $3$ and $4$, all quantities $s_k(T)$ can be expressed through $s_0(T)$ and the trivial value $s_{d-1}(T) = (d+1)/2$. In dimension $3$, the Gram--Euler relation states that $s_0(T) - s_1(T) + s_{2}(T) = 1$; see~\cite[Section~14.1]{GruenbaumBook}. In dimension $4$, we have the Gram--Euler relation  $s_0(T)-s_1(T)+s_2(T)-s_3(T) = -1$ and a Dehn--Sommerville relation
$-2 s_1(T) + 3 s_2(T) - 6 s_3(T) = -10$; see~\cite[Section 14.2 and p.~307]{GruenbaumBook}.
 If the  vertices are sampled uniformly on the sphere, we obtain the values
$$
\E s_1(T) =
\begin{cases}
\frac 98, &\text{ if } d=3,\\
\frac{539}{96\pi^2}, &\text{ if } d=4,
\end{cases}
\qquad
\E s_2(T) =
\begin{cases}
2, &\text{ if } d=3,\\
\frac{5}{3}+\frac{539}{144 \pi ^2}, &\text{ if } d=4.
\end{cases}
$$
For vertices sampled uniformly in the ball, we have
$$
\E s_1(T) =
\begin{cases}
\frac{2961}{2560}, &\text{ if } d=3,\\
\frac{1692197}{282240 \pi ^2}, &\text{ if } d=4,
\end{cases}
\qquad
\E s_2(T) =
\begin{cases}
2, &\text{ if } d=3,\\
\frac{5}{3}+\frac{1692197}{423360 \pi ^2}, &\text{ if } d=4.
\end{cases}
$$
The formulae remain valid if the uniform distribution on the ball is replaced by the uniform distribution on the interior of any $d$-dimensional ellipsoid~\cite{kabluchko_zaporozhets_ongoing}.

\subsection{Angle-sums of beta simplices}
In order to prove Theorems~\ref{theo:general_beta_d=3} and~\ref{theo:general_beta_d=4} it is  necessary to pass to a more general family of distributions including the uniform distributions on the ball and on the sphere as special cases. We say that a random vector in $\R^d$ has a $d$-dimensional \textit{beta distribution} with parameter $\beta>-1$ if its Lebesgue density is
\begin{equation}\label{eq:def_f_beta}
f_{d,\beta}(x)=c_{d,\beta} \left( 1-\left\| x \right\|^2 \right)^\beta\ind_{\{\|x\| <  1\}},\qquad x\in\R^d,\qquad
c_{d,\beta}= \frac{ \Gamma\left( \frac{d}{2} + \beta + 1 \right) }{ \pi^{ \frac{d}{2} } \Gamma\left( \beta+1 \right) }.
\end{equation}
Taking $\beta=0$, we recover the uniform distribution on the unit ball $\bB^d$.  The uniform distribution on the unit sphere $\bS^{d-1}$ appears as the weak limit of the beta distribution as $\beta\downarrow -1$; see~\cite[Proof of Corollary~3.9]{beta_polytopes_temesvari}.
These distributions were introduced by Ruben and Miles~\cite{ruben_miles} and Miles~\cite{miles}.  

Let $X_0,\ldots,X_d$ be independent random points in $\R^d$ distributed according to the beta distribution $f_{d,\beta}$, where $\beta\geq -1$. Their convex hull $[X_0,\ldots,X_d]$ is called the $d$-dimensional \textit{beta simplex}. We allow the value $\beta=-1$, in which case $X_0,\ldots,X_d$ are uniformly distributed on the unit sphere $\bS^{d-1}$. Beta simplices and, more generally, beta polytopes were studied in~\cite{ruben_miles,miles,beta_polytopes_temesvari,bonnet_etal,BonnetEtAlThresholds,beta_simplices,beta_polytopes}. In particular, it was demonstrated in~\cite{beta_polytopes} that many quantities appearing in stochastic geometry can be expressed through the expected internal angles of beta simplices, but no formula for the latter was obtained in~\cite{beta_polytopes} except for some trivial cases.

Now we are ready to state the results on the expected angle-sums of beta simplices in dimensions $3$ and $4$.



\begin{theorem}[$d=3$]\label{theo:general_beta_d=3}
Let $X_0,\ldots,X_3$ be i.i.d.\ points in the three-dimensional unit ball $\bB^3$ sampled from the distribution $f_{3,\beta}$, $\beta\geq -1$. Then, the expected sum of solid angles of the tetrahedron $T_{\beta}:=[X_0,\ldots,X_3]$ at its vertices is given by
\begin{equation*}
\E s_0(T_\beta)
=
2 - \frac{6\, \Gamma^2(\beta+\frac 52)\Gamma(2\beta+4)}{\pi^{3/2}\Gamma^2(\beta +2)\Gamma(2\beta + \frac 72)} \int_{-\pi/2}^{+\pi/2} (\cos\varphi)^{4\beta + 6} \left(\int_{-\pi/2}^\varphi (\cos \theta)^{2\beta+3}\dd \theta\right)^2 \dd \varphi.
\end{equation*}
\end{theorem}

\begin{theorem}[$d=4$]\label{theo:general_beta_d=4}
Let $X_0,\ldots,X_4$ be i.i.d.\ points in the four-dimensional unit ball $\bB^4$ with density $f_{4,\beta}$, $\beta\geq -1$. Then, the expected sum of solid angles of the $4$-dimensional simplex $T_\beta:=[X_0,\ldots,X_4]$ at its vertices is given by
\begin{equation*}
\E s_0(T_\beta)
=
\frac 32 -  \frac{5\, \Gamma^2(\beta+3)\Gamma(3\beta+7)}{\pi^{3/2} \Gamma^2(\beta + \frac 5 2)\Gamma(3\beta + \frac {13}2)}
\int_{-\pi/2}^{+\pi/2} (\cos\varphi)^{6\beta+12} \left(\int_{-\pi/2}^\varphi (\cos \theta)^{2\beta+4}\dd\theta\right)^2 \dd \varphi.
\end{equation*}
\end{theorem}

Theorems~\ref{theo:main_sphere} and~\ref{theo:main_ball} follow from Theorems~\ref{theo:general_beta_d=3} and~\ref{theo:general_beta_d=4} by taking $\beta=-1$ or $\beta=0$ and evaluating the integrals.
The inner integrals
are given by
\begin{align*}
\int_{-\pi/2}^\varphi \cos \theta \dd \theta &= 1+ \sin \varphi,\\
\int_{-\pi/2}^\varphi(\cos \theta)^2 \dd \theta &= \frac \varphi 2 + \frac \pi 4 + \frac 12 \cos \varphi \sin \varphi,\\
\int_{-\pi/2}^\varphi(\cos \theta)^3 \dd \theta &= -\frac{1}{12} \sin^3 \varphi+\frac{3}{4}\sin \varphi +\frac{1}{4} \sin\varphi \cos ^2\varphi  +\frac{2}{3},\\
\int_{-\pi/2}^\varphi(\cos \theta)^4 \dd \theta &= \frac {3\varphi}{8} + \frac{3\pi}{16} + \frac 12 \cos \varphi \sin \varphi + \frac 18 \cos^3 \varphi \sin \varphi  - \frac 18 \cos \varphi \sin^3\varphi.
\end{align*}
Then, the computation of $\E s_0(T_{-1})$ and $\E s_0(T_0)$ reduces to evaluating standard trigonometric integrals.

\section{Proofs of Theorems~\ref{theo:general_beta_d=3} and~\ref{theo:general_beta_d=4}}\label{sec:proofs}

We start by recalling some facts that will be needed in the proofs of Theorems~\ref{theo:general_beta_d=3} and~\ref{theo:general_beta_d=4}.

\subsection{Angles as probabilities}
The first ingredient in our proofs is the following elegant observation of Feldman and Klain~\cite{feldman_klain}. It can be viewed a special case of a more general result that has been obtained earlier by Affentranger and Schneider~\cite{AS92}.
\begin{theorem}[Feldman and Klain]\label{theo:feldman_klain}
Let $S=[x_0,\ldots,x_d]\subset \R^d$ be a $d$-dimensional simplex. Let $U$ be a random vector uniformly distributed on the unit sphere $\bS^{d-1}$ and denote by $\Pi=\Pi_{U^\bot}$ the orthogonal projection onto the orthogonal complement of $U$.
Then, the sum of solid angles at all vertices of $S$ satisfies
$$
2 s_0(S) = \P[\, \Pi S \text{ is a $(d-1)$-dimensional simplex} \,].
$$
\end{theorem}
Indeed, if we ignore degenerate cases of probability $0$, then the projection $\Pi S$ is a $(d-1)$-dimensional simplex if and only if the projection of one of the vertices of $S$ falls into the convex hull of the projections of the other vertices.
Such vertex is unique, if it exists. Consider, for concreteness, the random  event $\{\Pi x_0 \in [\Pi x_1,\ldots,\Pi x_d]\}$. It occurs  if and only if $U$ or $-U$ belongs to the tangent cone of $S$ at the vertex $x_0$ defined by
$$
T(x_0,P) := \{y\in \R^d\colon  \exists \eps>0 \text{ such that } x_0 + \eps y \in P\}.
$$
The probability of this event is twice the solid angle of $S$ at $x_0$. Taking the sum over all vertices $x_i$, Feldman and Klain  arrived at their formula.

\subsection{Number of facets of the beta polytope}
The second ingredient is the following formula for the expected number of facets of the beta polytope essentially obtained in~\cite{beta_polytopes_temesvari}. The convex hull of $n$ independent random points $Z_1,\ldots,Z_n$ in $\R^d$ distributed according to the law $f_{d,\beta}$ with $\beta\geq -1$ is called the \textit{beta polytope} and denoted by
$$
P_{n,d}^\beta := [Z_1,\ldots,Z_n].
$$
Let $f_{k}(P_{n,d}^\beta)$ be the number of $k$-dimensional faces of the polytope $P_{n,d}^\beta$.

\begin{theorem}[Facets of the beta polytope]\label{theo:facets_beta_polytope}
The expected number of facets of $P_{n,d}^\beta$ is given by
\begin{equation}
\E f_{d-1}(P_{n,d}^\beta)
=
C_{n,d}^{\beta,0} \int_{-1}^{1} \left(1-h^2\right)^{d\beta + \frac{d^2-1}{2}} (F_{1,\beta+\frac{d-1}{2}}(h))^{n-d} \dd h,
\end{equation}
where
\begin{align}
F_{1,\beta}(h)
&=
c_{1,\beta} \int_{-1}^h (1-x^2)^\beta \dd x, \qquad h\in [-1,1],\label{eq:F_def}\\
C_{n,d}^{\beta,0}
&=
\binom nd \cdot \frac 2 {\Gamma(\frac d2)} \cdot \frac{\Gamma\left( \frac{d}{2}(2 \beta + d) +1 \right)}{\Gamma\left( \frac{d}{2}(2 \beta + d) + \frac{1}{2} \right)} \cdot \prod_{i=1}^{d-1} \frac{\Gamma\left( \frac{i+1}{2} \right)}{\Gamma\left( \frac{i}{2} \right)}.
\label{eq:facets_beta_poly_C_n_d_beta_0}
\end{align}
\end{theorem}
\begin{proof}
By Remark 2.14 and Theorem 2.11 of~\cite{beta_polytopes_temesvari} we have (using the notation of that paper)
$$
\E f_{d-1}(P_{n,d}^\beta) =  \E T^{d,d-1}_{0,0} (P_{n,d}^\beta)
=
C_{n,d}^{\beta,0} \int_{-1}^{1} \left(1-h^2\right)^{d\beta + \frac{d^2-1}{2}} (F_{1,\beta+\frac{d-1}{2}}(h))^{n-d} \dd h,
$$
where
$$
C_{n,d}^{\beta,0}
=
\binom{n}{d} d! \frac{\pi^{d/2}}{\Gamma(\frac d2 +1)} \E_\beta \left( \Delta_{d-1} \right) \left( \frac{c_{d,\beta}}{c_{d-1,\beta}} \right)^d.
$$
The value of $\E_\beta( \Delta_{d-1})$ is given by Proposition 2.8 (a) of~\cite{beta_polytopes_temesvari} as follows:
$$
\E_\beta (\Delta_{d-1}^1)
=
\frac 1 {(d-1)!} \frac{\Gamma\left( \frac{d}{2}(2 \beta + d) +1 \right)}{\Gamma\left( \frac{d}{2}(2 \beta + d ) + \frac{1}{2} \right)} \left( \frac{\Gamma\left( \frac{d+1}{2} + \beta\right)}{\Gamma\left( \frac{d}{2} + \beta + 1 \right)} \right)^{d}  \cdot \prod_{i=1}^{d-1} \frac{\Gamma\left( \frac{i+1}{2} \right)}{\Gamma\left( \frac{i}{2} \right)}.
$$
Taking everything together, we obtain the required formula for $C_{n,d}^{\beta,0}$.
\end{proof}

\subsection{Projections of beta distributions}
Denote by $\pi_L: \R^d\rightarrow L$ the orthogonal projection on a $(d-1)$-dimensional linear subspace $L\subset \R^d$ which is allowed to be random.
The next result, see~\cite[Lemma~4.4]{beta_polytopes_temesvari}, states essentially that the projection of the $f_{d,\beta}$-distribution to $L$ is the $f_{d-1,\beta+\frac 12}$-distribution. There is, however, one technical subtlety to take care of. The projected distribution is a probability measure on $L$ (which is random) rather than on $\R^{d-1}$, hence we need to fix some way of identifying $L$ with the standard Euclidean space $\R^{d-1}$. To this end, we fix for each linear hyperplane $H\subset \R^d$ an isometry $I_H: H \to \R^{d-1}$ identifying $H$ with $\R^{d-1}$ such that $I_H(0)=0$. The only requirement we impose on this family of isometries is the Borel-measurability of the map $(x,H) \mapsto I_H(\pi_H(x))$.

\begin{lemma}[Orthogonal projections]\label{lem:projection}
Let $L$ be a random (not necessarily uniformly distributed)  $(d-1)$-dimensional linear subspace of $\R^d$.
If the random point $X$ has distribution $f_{d,\beta}$ for some $\beta\geq -1$ and is independent of $L$, then $I_L(\pi_L (X))$ has density $f_{d-1,\beta+\frac{1}{2}}$.
\end{lemma}
In~\cite[Lemma~4.4]{beta_polytopes_temesvari}, the lemma was proved for the hyperplane spanned by the first $d-1$ standard basis vectors. By rotational invariance of the beta distribution, it is true for an arbitrary deterministic hyperplane $L$. To prove it for random $L$, condition on all possible realizations of $L$ and integrate.

\subsection{Proof of Theorem~\ref{theo:general_beta_d=3}}
Recall that $X_0,\ldots,X_3$ are independent random points in the three-dimensional unit ball $\bB^3$  sampled from the distribution $f_{3,\beta}$, $\beta\geq -1$. Independently of the $X_i$'s, let $U$ be uniformly distributed on the unit sphere $\bS^{2}$.
Consider an orthogonal projection $\Pi$ of the tetrahedron $[X_0,\ldots,X_3]$ onto the random, uniformly distributed, two-dimensional plane $L:= U^\bot$.  By the projection property of the beta densities stated in Lemma~\ref{lem:projection}, the four projected points (or, to be more precise, the points $I_L(\Pi X_i)$, $i=0,\ldots,3$) have the density $f_{2,\beta+\frac 12}$. It follows that
$$
\E f_1 (\Pi [X_0,\ldots,X_3]) = \E f_1(P_{4,2}^{\beta+ \frac 12}).
$$

Disregarding degenerate cases (that have probability $0$), we have two possibilities: either the projection is a triangle, or the projection is a quadrilateral. Denote the probability that the projection is a triangle by $p$. Then, by Theorem~\ref{theo:feldman_klain},
$$
\E s_0(T_\beta) = p/2.
$$
On the other hand, we can compute the expected number of edges of the projection as follows:
$$
\E f_1 (\Pi [X_0,\ldots,X_3]) = 4(1-p) + 3 p = 4 - p.
$$
It follows that
$$
\E s_0(T_\beta) = 2 - \frac 12 \E f_1(P_{4,2}^{\beta+ \frac 12}).
$$
By Theorem~\ref{theo:facets_beta_polytope}, we have
$$
\E f_1(P_{4,2}^{\beta+ \frac 12})
=
C_{4,2}^{\beta+\frac 12,0} \int_{-1}^{+1} (1-h^2)^{2\beta + \frac 52} \left(c_{1,\beta+1} \int_{-1}^h (1-x^2)^{\beta+1} \dd x\right)^2 \dd h.
$$
Here, $C_{4,2}^{\beta+\frac 12,0}$ is the constant given by~\eqref{eq:facets_beta_poly_C_n_d_beta_0}. After some algebra, we obtain
$$
C_{4,2}^{\beta+\frac 12,0} = \frac{12}{\sqrt \pi} \cdot \frac{\Gamma(2\beta+4)}{\Gamma(2\beta + \frac 72)}.
$$
Taking everything together, we arrive at
\begin{equation}\label{eq:E_s_0_d_3}
\E s_0(T_\beta)
=
2 - \frac{6}{\sqrt \pi} \cdot \frac{\Gamma(2\beta+4)}{\Gamma(2\beta + \frac 72)} \cdot\int_{-1}^{+1} (1-h^2)^{2\beta + \frac 52} \left(c_{1,\beta+1} \int_{-1}^h (1-x^2)^{\beta+1} \dd x\right)^2 \dd h.
\end{equation}
Recall from~\eqref{eq:def_f_beta} that
$$
c_{1,\beta+1} = \frac{ \Gamma\left(\beta + \frac 52 \right) }{ \sqrt \pi \Gamma\left( \beta+2\right) }.
$$
The change of variables $h= \sin \varphi$ and $x=\sin \theta$ with $\varphi,\theta\in (-\pi/2, +\pi/2)$ transforms~\eqref{eq:E_s_0_d_3} into
$$
\E s_0(T_\beta)
=
2 - \frac{6\, \Gamma^2(\beta+\frac 52)\Gamma(2\beta+4)}{\pi^{3/2}\Gamma^2(\beta +2)\Gamma(2\beta + \frac 72)} \int_{-\pi/2}^{+\pi/2} (\cos\varphi)^{4\beta + 6} \left(\int_{-\pi/2}^\varphi (\cos \theta)^{2\beta+3}\dd \theta\right)^2 \dd \varphi.
$$
This completes the proof. \hfill $\Box$

\subsection{Proof of Theorem~\ref{theo:general_beta_d=4}}
Recall that $X_0,\ldots,X_4$ are independent random points in the four-dimensional unit ball $\bB^4$  sampled from the probability distribution $f_{4,\beta}$, $\beta\geq -1$. Independently of the $X_i$'s, let $U$ be uniformly distributed on the unit sphere $\bS^{3}$.
Let $\Pi$ denote the orthogonal projection on  the random, uniformly distributed, three-dimensional linear subspace $L:= U^\bot$.  By the projection property of the beta densities stated in Lemma~\ref{lem:projection}, the points $I_L(\Pi X_i)$, $i=0,\ldots,4$, have the density $f_{3,\beta+\frac 12}$ on $\R^3$. It follows that
$$
\E f_2 (\Pi [X_0,\ldots,X_4]) = \E f_2(P_{5,3}^{\beta+ \frac 12}).
$$
Denote by $p$ the probability that the projection $\Pi [X_0,\ldots,X_4]$  is a $3$-dimensional simplex. Then, by Theorem~\ref{theo:feldman_klain},
$$
\E s_0(T_\beta) = p/2.
$$
Disregarding degenerate cases of probability $0$, there are two possible combinatorial types of the projected polytope: the $3$-dimensional tetrahedron with $4$ facets, and the $3$-dimensional polytope with $6$ facets obtained by gluing together two tetrahedra at a common facet; see~\cite[Chapter~6.1]{GruenbaumBook}. Since the probabilities of these combinatorial types are $p$ and $1-p$, respectively, we have
$$
\E f_2 (P_{5,3}^{\beta + \frac 12}) = 4 p + 6 (1-p) = 6-2p.
$$
It follows that
$$
\E s_0(T_\beta) = \frac 32 - \frac 14 \E f_2 (P_{5,3}^{\beta + \frac 12}).
$$
The expected number of facets of $P_{5,3}^{\beta + \frac 12}$ can be computed with the help of Theorem~\ref{theo:facets_beta_polytope} as follows:
$$
\E f_2 (P_{5,3}^{\beta + \frac 12}) =  C_{5,3}^{\beta + \frac 12,0}  \int_{-1}^{+1} (1-h^2)^{3\beta + \frac {11} 2}  \left(c_{1,\beta+\frac 32} \int_{-1}^h (1-x^2)^{\beta+\frac 32} \dd x\right)^2 \dd h.
$$
The value of the constant $C_{5,3}^{\beta + \frac 12,0}$ is given by~\eqref{eq:facets_beta_poly_C_n_d_beta_0}.  After some simplifications we arrive at
$$
C_{5,3}^{\beta + \frac 12,0} = \frac {20}{\sqrt \pi} \cdot \frac{\Gamma(3\beta+7)}{\Gamma(3\beta + \frac {13}2)}.
$$
Taking everything together leads to
\begin{equation}\label{eq:E_s_0_d_4}
\E s_0(T_\beta)  = \frac 32  -\frac {5}{\sqrt \pi} \cdot \frac{\Gamma(3\beta+7)}{\Gamma(3\beta + \frac {13}2)} \cdot
 \int_{-1}^{+1} (1-h^2)^{3\beta + \frac {11} 2}  \left(c_{1,\beta+\frac 32} \int_{-1}^h (1-x^2)^{\beta+\frac 32} \dd x\right)^2 \dd h.
\end{equation}
Recall from~\eqref{eq:def_f_beta} that
$$
c_{1,\beta+\frac 32} = \frac{ \Gamma\left(\beta + 3 \right) }{ \sqrt \pi \Gamma\left( \beta + \frac 5 2\right) }.
$$
The change of variables $h= \sin \varphi$ and $x=\sin \theta$ transforms~\eqref{eq:E_s_0_d_4} into
$$
\E s_0(T_\beta) = \frac 32 -  \frac{5\, \Gamma^2(\beta+3)\Gamma(3\beta+7)}{\pi^{3/2} \Gamma^2(\beta + \frac 5 2)\Gamma(3\beta + \frac {13}2)}
\int_{-\pi/2}^{+\pi/2} (\cos\varphi)^{6\beta+12} \left(\int_{-\pi/2}^\varphi (\cos \theta)^{2\beta+4}\dd\theta\right)^2 \dd \varphi.
$$
The proof is complete. \hfill $\Box$

\subsection{Remark on higher dimensions}
The method that we used in the cases $d=3$ and $d=4$ breaks down for $d\geq 5$ for the following reason. If $d = 5$, then the projection of the beta-simplex on a random hyperplane is a $4$-dimensional simplicial polytope with at most $6$ vertices. It is known~\cite[Chapter~6.1]{GruenbaumBook} that there are three different combinatorial types of such polytopes: the simplex and two polytopes denoted by $T_1^4$ and $T_2^4$. The method used in the present paper is based on the fact that for $d=3,4$ there are just two possible combinatorial types (or a weaker statement that all combinatorial types except the simplex have the same number of facets). Even the latter weaker statement breaks down for $d=5$ since $T_1^4$ and $T_2^4$ have $8$ and $9$ facets, respectively.   For general $d\geq 5$, the situation gets even worse since the number of different combinatorial types is then $1+ [(d-1)/2]$; see~\cite[Chapter~6.1]{GruenbaumBook}.

\section*{Acknowledgement}
Supported by the German Research Foundation under Germany's Excellence Strategy  EXC 2044 -- 390685587, Mathematics M\"unster: Dynamics - Geometry - Structure.


\bibliography{angles_bib}

\begin{thebibliography}{14}
\providecommand{\natexlab}[1]{#1}
\providecommand{\url}[1]{\texttt{#1}}
\expandafter\ifx\csname urlstyle\endcsname\relax
  \providecommand{\doi}[1]{doi: #1}\else
  \providecommand{\doi}{doi: \begingroup \urlstyle{rm}\Url}\fi

\bibitem[Affentranger and Schneider(1992)]{AS92}
F.~Affentranger and R.~Schneider.
\newblock Random projections of regular simplices.
\newblock \emph{Discrete Comput. Geom.}, 7\penalty0 (1):\penalty0 219--226,
  1992.
\newblock \doi{10.1007/BF02187839}.

\bibitem[Bonnet et~al.(2017)Bonnet, Grote, Temesvari, Th\"ale, Turchi, and
  Wespi]{bonnet_etal}
G.~Bonnet, J.~Grote, D.~Temesvari, C.~Th\"ale, N.~Turchi, and F.~Wespi.
\newblock Monotonicity of facet numbers of random convex hulls.
\newblock \emph{J. Math. Anal. Appl.}, 455\penalty0 (2):\penalty0 1351--1364,
  2017.

\bibitem[{Bonnet} et~al.(2018){Bonnet}, {Chasapis}, {Grote}, {Temesvari}, and
  {Turchi}]{BonnetEtAlThresholds}
G.~{Bonnet}, G.~{Chasapis}, J.~{Grote}, D.~{Temesvari}, and N.~{Turchi}.
\newblock {Threshold phenomena for high-dimensional random polytopes}.
\newblock \emph{ArXiv e-prints}, February 2018.

\bibitem[Feldman and Klain(2009)]{feldman_klain}
D.~V. Feldman and D.~A. Klain.
\newblock Angles as probabilities.
\newblock \emph{Amer. Math. Monthly}, 116\penalty0 (8):\penalty0 732--735,
  2009.
\newblock \doi{10.4169/193009709X460868}.
\newblock URL \url{https://doi.org/10.4169/193009709X460868}.

\bibitem[Grote et~al.(2019)Grote, Kabluchko, and Th\"{a}le]{beta_simplices}
J.~Grote, Z.~Kabluchko, and C.~Th\"{a}le.
\newblock Limit theorems for random simplices in high dimensions.
\newblock \emph{ALEA Lat. Am. J. Probab. Math. Stat.}, 16\penalty0
  (1):\penalty0 141--177, 2019.

\bibitem[Gr\"unbaum(2003)]{GruenbaumBook}
B.~Gr\"unbaum.
\newblock \emph{Convex {P}olytopes}, volume 221 of \emph{Graduate Texts in
  Mathematics}.
\newblock Springer-Verlag, New York, second edition, 2003.
\newblock \doi{10.1007/978-1-4613-0019-9}.
\newblock Prepared and with a preface by V. Kaibel, V. Klee and G. M. Ziegler.

\bibitem[Kabluchko and Zaporozhets()]{kabluchko_zaporozhets_ongoing}
Z.~Kabluchko and D.~Zaporozhets.
\newblock Work in progress.

\bibitem[Kabluchko and Zaporozhets(2018)]{kabluchko_zaporozhets_gauss_simplex}
Z.~Kabluchko and D.~Zaporozhets.
\newblock Angles of the {G}aussian simplex.
\newblock \emph{Zap. Nauchn. Sem. POMI}, 476:\penalty0 79--91, 2018.
\newblock Preprint at arXiv: 1801.08008.

\bibitem[Kabluchko et~al.(2018)Kabluchko, Th{\"a}le, and
  Zaporozhets]{beta_polytopes}
Z.~Kabluchko, C.~Th{\"a}le, and D.~Zaporozhets.
\newblock Beta polytopes and {P}oisson polyhedra: {$f$}-vectors and angles.
\newblock Preprint at http://arxiv.org/abs/1805.01338, 2018.

\bibitem[Kabluchko et~al.(2019)Kabluchko, Temesvari, and
  Th\"{a}le]{beta_polytopes_temesvari}
Z.~Kabluchko, D.~Temesvari, and C.~Th\"{a}le.
\newblock Expected intrinsic volumes and facet numbers of random
  beta-polytopes.
\newblock \emph{Math. Nachr.}, 292\penalty0 (1):\penalty0 79--105, 2019.
\newblock \doi{10.1002/mana.201700255}.
\newblock URL \url{https://doi.org/10.1002/mana.201700255}.

\bibitem[{Miles}(1971)]{miles}
R.E. {Miles}.
\newblock {Isotropic random simplices.}
\newblock \emph{{Adv. Appl. Probab.}}, 3:\penalty0 353--382, 1971.
\newblock \doi{10.2307/1426176}.

\bibitem[Perles and Shephard(1967)]{perles_shephard}
M.~A. Perles and G.~C. Shephard.
\newblock Angle sums of convex polytopes.
\newblock \emph{Math. Scand.}, 21:\penalty0 199--218 (1969), 1967.
\newblock \doi{10.7146/math.scand.a-10860}.
\newblock URL \url{https://doi.org/10.7146/math.scand.a-10860}.

\bibitem[{Ruben} and {Miles}(1980)]{ruben_miles}
H.~{Ruben} and R.E. {Miles}.
\newblock {A canonical decomposition of the probability measure of sets of
  isotropic random points in $\mathbb R^n$.}
\newblock \emph{{J. Multivariate Anal.}}, 10:\penalty0 1--18, 1980.
\newblock \doi{10.1016/0047-259X(80)90077-9}.

\bibitem[Schneider and Weil(2008)]{SW08}
R.~Schneider and W.~Weil.
\newblock \emph{Stochastic and {I}ntegral {G}eometry}.
\newblock Probability and its Applications. Springer--Verlag, Berlin, 2008.

\end{thebibliography}
\bibliographystyle{plainnat}

\end{document}